\documentclass[12pt]{article}
\usepackage[a4paper, total={15cm, 21cm}]{geometry}
\usepackage[T1]{fontenc}
\usepackage{amsmath}
\usepackage{amsfonts}
\usepackage{amssymb}
\usepackage{amsthm}
\usepackage{stmaryrd}
\usepackage{graphicx}
\usepackage{tikz}
\usepackage{tikz-cd}
\usepackage{xcolor}
\usepackage{nicefrac}
\usepackage{hyperref}

\newtheorem{theorem}{Theorem}[section]
\newtheorem{proposition}[theorem]{Proposition}

\newtheorem{lemma}[theorem]{Lemma}

\theoremstyle{definition}

\newtheorem{remark}[theorem]{Remark}

\newcommand{\Ps}{\mathbb{P}}

\newcommand{\Z}{\mathbb{Z}}
\newcommand{\R}{\mathbb{R}}
\newcommand{\N}{\mathbb{N}}
\newcommand{\K}{\mathbb{K}}

\newcommand{\End}{\mathrm{End}}

\title{A uniform Tits alternative for endomorphisms of the projective line}
\author{Alonso Beaumont\thanks{IRMAR - UMR CNRS 6625, Université de Rennes. E-mail: alonso.beaumont@univ-rennes.fr. Research supported by the European Research Council (ERC Groups of Algebraic Transformations 101053021).}} 
\date{January 2025}

\begin{document}

\maketitle

\begin{abstract}
    The recent article \cite{BHPT24} establishes an analog of the Tits alternative for semigroups of endomorphisms of the projective line. The proof involves a ping-pong argument on arithmetic height functions. Extending this method, we obtain a uniform version of the same alternative. In particular, we show that semigroups of $\End(\Ps^{1})$ of exponential growth are of uniform exponential growth.
\end{abstract}

\section{Introduction}

Let $\K$ be an algebraically closed field and let $V$ be a projective variety over $\K$. We denote by $\End(V)$ the set of endomorphisms of $V$ defined over $\K$, which has the structure of a semigroup when endowed with the composition operation. Which types of growth rate can occur for finitely generated subsemigroups of $\End(V)$, and how to characterize semigroups of a given growth rate? This is a natural generalization of the study of finitely generated linear semigroups, an overview of which can be found in \cite{Okn98}.\bigskip

For any $f\in \End(V)$, let $\mathrm{PrePer}(f)$ be its set of \textit{preperiodic points}, that is, the set of points $z\in V(\K)$ whose orbits under the iteration of $f$ are finite. An endomorphism $f\in \End(V)$ is said to be \textit{polarized} by an ample line bundle $\mathcal{L}$ if $f^{*}\mathcal{L}\cong\mathcal{L}^{\otimes d}$ for some $d\in\Z_{\geq2}$. The integer $d$ is called the \textit{algebraic degree} of $f$. The notion of polarized endomorphisms was introduced in \cite{Zha95}. Such an endomorphism is also finite (see \cite{Ser60}, and \cite[§5]{Fak03} for the case of positive characteristic).\bigskip

The purpose of this note is to prove the following result:

\begin{theorem}
    \label{uni1}
    Let $f_{1},f_{2}\in\End(V)$ be polarized by the same line bundle and suppose that $\mathrm{PrePer}(f_{1})\neq\mathrm{PrePer}(f_{2})$. Then $f_{1}$ and $f_{2}$ generate a free semigroup of rank $2$.
\end{theorem}

This theorem is proven using a ping-pong argument on \textit{height functions}, which will be introduced in Section \ref{sec2}. The argument used is essentially a refinement of the methods in \cite[§3]{BHPT24}, and answers Question 5.1 of the same article.\bigskip

Theorem \ref{uni1} allows us to obtain results concerning uniform independence and uniform exponential growth for certain subsemigroups of $\End(V)$. These results are best stated using the following definitions:

Let $S$ be any semigroup. Two elements of $S$ are said to be \textit{independent} if they generate a free semigroup of rank $2$. For any non-empty finite subset $F\subset S$, we define the \textit{diameter of independence} of $F$ as
\[\Delta(F):=\inf\;\{n\in\Z_{\geq1}\;|\; F\cup F^{2}\cup\cdots\cup F^{n} \textrm{ contains two independent elements}\}\]
where $F^{n}=\{f_{1}\cdots f_{n};\; f_{i}\in F\}$. The \textit{algebraic entropy} of $F$ is defined as
\[\Sigma(F):=\lim_{n\rightarrow\infty}\frac{1}{n}\log\#(F\cup F^{2}\cup\cdots\cup F^{n}).\]
Note that $\Delta(F)$ may be infinite, whereas $\Sigma(F)$ is always finite. Moreover, we have the following inequality:
\begin{equation}
    \label{eq1}
    \Sigma(F)\geq \log(2)/\Delta(F). 
\end{equation}
If $S$ is finitely generated, its diameter of independence and its algebraic entropy are respectively defined as
\[\Delta(S):=\sup_{F}\Delta(F)\quad\mathrm{and}\quad\Sigma(S):=\inf_{F}\Sigma(F)\]
where $F$ ranges over all finite generating sets of $S$. The semigroup $S$ is of \textit{exponential growth} if $\Sigma(F)>0$ for some (equivalently, for any) finite generating set $F$ of $S$; it is of \textit{uniform exponential growth} if $\Sigma(S)>0$.\bigskip

In order to accomodate for automorphisms, we introduce the following notion: $f\in\End(V)$ is \textit{semi-polarized} by $\mathcal{L}$ if $f^{*
}\mathcal{L}\cong\mathcal{L}^{\otimes d}$ for some $d\in\mathbb{Z}_{\geq1}$. In particular, we allow $d=1$.

\begin{theorem}
    \label{uni2}
    Let $V$ be a projective variety over an algebraically closed field $\K$. Let $S$ be a finitely generated subsemigroup of endomorphsims of $V$, all semi-polarized by the same line bundle. Suppose there are $f_{1},f_{2}\in S$ of algebraic degree at least $2$ such that $\mathrm{PrePer}(f_{1})\neq\mathrm{PrePer}(f_{2})$. Then $\Delta(S)\leq2$. In particular, $S$ is of uniform exponential growth: $\Sigma(S)\geq\log(2)/2$.
\end{theorem}

In the case where $\mathrm{char(\K)=0}$ and $V=\Ps^{1}$, we may combine Theorem \ref{uni2} with \cite[Proposition 4.10]{BHPT24}, along with results of E. Breuillard and T. Gelander \cite{BG05} in the linear case, to obtain the following

\begin{theorem}
    \label{uni3}
    Let $S$ be a finitely generated subsemigroup of $\End(\Ps^{1})$ over a field of characteristic $0$. Then either $S$ is of polynomial growth, or $\Delta(S)<+\infty$. In particular, semigroups of exponential growth in $\End(\Ps^{1})$ are of uniform exponential growth.
\end{theorem}

\section{A ping-pong lemma for contractions}
\label{sec1}

The notions introduced in this section are presented in greater detail in the book \cite{Fal85}. The setting of the book is Euclidean space, but the results still hold in the generality in which we use them.\bigskip

Let $(X,d)$ be a non-empty complete metric space, in which we fix a base point $x_{0}$. A map $\alpha:X\rightarrow X$ is called a \textit{contraction with ratio} $c\in]0,1[$ if
\[\forall x,y\in X,\;d(\alpha(x),\alpha(y))\leq c\cdot d(x,y).\]
By Banach's fixed point theorem, $\alpha$ has a unique fixed point in $X$. We denote by $\mathrm{Con}(X)$ the set of contractions on $X$, a semigroup when endowed with the composition operation. For any subset $F\subset \mathrm{Con}(X)$, we denote by $\langle F\rangle$ the subsemigroup generated by $F$.

\paragraph{The attractor.} Let $\alpha_{1},\alpha_{2}$ be two contractions with ratios $c_{1},c_{2}$, and let $C=\{\alpha_{1},\alpha_{2}\}^{\N}$. For any sequence $u=(\alpha_{n_{i}})_{i\geq0}\in C$, the sequence $((\alpha_{n_{0}}\cdots\alpha_{n_{i}})(x_{0}))_{i\geq0}$ is Cauchy, and therefore converges to an element $x_{u}$ in $X$. If we endow $C$ with the product topology, the map
\begin{equation}
    \label{eq2}
    \pi:C\rightarrow X,\;u\mapsto x_{u}
\end{equation}
is continuous: its image $A$ is therefore a compact subset of $X$. The set $A$ is called the \textit{attractor} associated with the system $\{\alpha_{1},\alpha_{2}\}$. For $i\in\{1,2\}$, we have a commutative diagram
\[\begin{tikzcd}
    C && C \\
    A && A
    \arrow["{\sigma_{i}}", from=1-1, to=1-3]
    \arrow["{\pi}", from=1-1, to=2-1]
    \arrow["{\pi}", from=1-3, to=2-3]    \arrow["{\alpha_{i}}", from=2-1, to=2-3]
\end{tikzcd}\]
where $\sigma_{i}:(\alpha_{n_{0}},\alpha_{n_{1}},\cdots)\mapsto(\alpha_{i},\alpha_{n_{0}},\alpha_{n_{1}},\cdots)$. Since $C=\sigma_{1}(C)\cup\sigma_{2}(C)$, $A$ satisfies the self-similarity relation $A=\alpha_{1}(A)\cup\alpha_{2}(A)$. It is in fact the unique non-empty compact set that satisfies this relation (see \cite[§8.3]{Fal85}).\bigskip

For any subset $Y\subset X$, we denote its diameter by $\mathrm{diam}(Y)$. The \textit{one-dimensional Hausdorff measure} of $Y$ is defined as follows (see \cite[§1.2]{Fal85}):
\[H^{1}(Y)=\lim_{\varepsilon\rightarrow 0}\;\inf\left\{\sum_{i}\mathrm{diam}(U_{i})\;\middle|\; Y\subset\bigcup_{i}U_{i}\;,\; 0<\mathrm{diam}(U_{i})\leq\varepsilon\right\}\]

\begin{proposition}
    \label{pingpong}
    Let $\alpha_{1},\alpha_{2}:X\rightarrow X$ be two injective contractions with ratios $c_{1},c_{2}\in]0,1[$. Suppose that $\alpha_{1}$ and $\alpha_{2}$ have distinct fixed points and $c_{1}+c_{2}\leq1$. Then $\langle\alpha_{1},\alpha_{2}\rangle$ is a free semigroup of rank $2$.
\end{proposition}

\begin{proof}
    Let $A$ be the attractor associated with $\{\alpha_{1},\alpha_{2}\}$ and $\pi:C\rightarrow A$ be the map introduced in equation (\ref{eq2}). The first step of the proof is adapted from \cite{BH85}.
    
    \paragraph{Step 1.} Suppose that $A$ is disconnected. We will prove that $\alpha_{1}(A)\cap\alpha_{2}(A)=\varnothing$ and conclude that $\langle\alpha_{1},\alpha_{2}\rangle$ is free of rank 2.
    
    Since $A$ is disconnected, there are non-empty compact subsets $A_{1},A_{2}$ such that $A_{1}\cup A_{2}=A$ and $A_{1}\cap A_{2}=\varnothing$. By the compactness of $A$,
    \[\inf\;\{d(x_{1},x_{2});\;x_{1}\in A_{1},\;x_{2}\in A_{2}\}=:\delta>0.\]
    For $u_{1},u_{2}\in C$, let $\lambda(u_{1},u_{2})$ be their largest common prefix, and consider the set $\mathcal{P}=\{\lambda(u_{1},u_{2});\;(u_{1},u_{2})\in \pi^{-1}(A_{1})\times \pi^{-1}(A_{2})\}$. If we choose an integer $n\geq0$ such that $\delta\geq (\max(c_{1},c_{2}))^{n}\cdot\mathrm{diam}(A)$, then the words in $\mathcal{P}$ are of length at most $n$. We may therefore choose a prefix $p\in \mathcal{P}$ of maximal length, as well as a pair $(pu_{1},pu_{2})\in \pi^{-1}(A_{1})\times \pi^{-1}(A_{2})$. If there existed an element $y\in\alpha_{1}(A)\cap\alpha_{2}(A)$, say
    \[y=\pi(\alpha_{1}v_{1})=\pi(\alpha_{2}v_{2}),\]
    then without loss of generality, $p(y)=\pi(p\alpha_{1}v_{1})=\pi(p\alpha_{2}v_{2})\in A_{1}$, and $pu_{2}$ would have a common prefix with either $p\alpha_{1}v_{1}$ or $p\alpha_{2}v_{2}$ which would be longer than $p$, contradicting its maximality. We deduce that $\alpha_{1}(A)\cap\alpha_{2}(A)=\varnothing$. 
    
    This allows us to perform a ping-pong argument. Suppose there is a pair of distinct words $w_{1},w_{2}$ in the alphabet $\{\alpha_{1},\alpha_{2}\}$ such that $w_{1}=w_{2}$ in $\langle\alpha_{1},\alpha_{2}\rangle$. Since $\alpha_{1}$ and $\alpha_{2}$ are injective, left cancellation implies that $\alpha_{1}w_{1}'=\alpha_{2}w_{2}'$ in $\langle\alpha_{1},\alpha_{2}\rangle$ for some pair of words $w_{1}',w_{2}'$. But then
    \[\alpha_{1}w_{1}'(A)=\alpha_{1}w_{1}'(A)\cap\alpha_{2}w_{2}'(A)\subset\alpha_{1}(A)\cap\alpha_{2}(A)=\varnothing\]
    which is absurd. We conclude that $\langle\alpha_{1},\alpha_{2}\rangle$ is free of rank $2$.
    
    \paragraph{Step 2.} Suppose now that $A$ is connected. We will prove that $H^{1}(A)=\mathrm{diam}(A)$ and that $\alpha_{1}(A)$ and $\alpha_{2}(A)$ have an intersection of zero $H^{1}$-measure. This will allow us to perform a measurable version of the ping-pong argument above.
    
    By the self-similarity relation that characterizes the attractor, we have for all $n\geq0$,
    \[\bigcup_{w\in\{\alpha_{1},\alpha_{2}\}^{n}}w(A)=A.\]
    Since $\alpha_{1}$ and $\alpha_{2}$ are contractions, we have
    \[\mathrm{diam}(\alpha_{1}(A))+\mathrm{diam}(\alpha_{2}(A))\leq c_{1}\cdot\mathrm{diam}(A)+c_{2}\cdot\mathrm{diam}(A)=(c_{1}+c_{2})\cdot\mathrm{diam}(A)\]
    and more generally, for all $n\geq0$,
    \[\sum_{w\in\{\alpha_{1},\alpha_{2}\}^{n}}\mathrm{diam}(w(A))\leq(c_{1}+c_{2})^{n}\cdot\mathrm{diam}(A).\]
    This construction provides arbitrarily fine covers of $A$, so that $H^{1}(A)\leq\mathrm{diam}(A)$, and $H^{1}(A)=0$ whenever $c_{1}+c_{2}<1$.
    
    Let $x,y\in A$ be such that $d(x,y)=\mathrm{diam}(A)$. The map
    \[A\rightarrow[0,\mathrm{diam}(A)],\;z\mapsto d(x,z)\]
    is surjective because it is continuous and $A$ is connected. Moreover, this map is also $1$-Lipschitz, so $\mathrm{diam}(A)=H^{1}([0,\mathrm{diam}(A)])\leq H^{1}(A)$ (\cite[Lemma 1.8]{Fal85}). Note that $\mathrm{diam}(A)>0$ because the fixed points of $\alpha_{1}$ and $\alpha_{2}$ are distinct. In view of the upper bounds on $H^{1}(A)$ obtained above, we must have $c_{1}+c_{2}=1$ and $H^{1}(A)=\mathrm{diam}(A)$. Finally,
    \[H^{1}(A)\leq H^{1}(\alpha_{1}(A))+H^{1}(\alpha_{2}(A))\leq c_{1}H^{1}(A)+c_{2}H^{1}(A)=H^{1}(A)\]
    so $H^{1}(\alpha_{1}(A)\cap\alpha_{2}(A))=0$.
    
    Suppose there is a distinct pair of words $w_{1},w_{2}$ in the alphabet $\{\alpha_{1},\alpha_{2}\}$ such that $w_{1}=w_{2}$ in $\langle\alpha_{1},\alpha_{2}\rangle$. As in the first step, we may assume that $w_{1}=\alpha_{1}w_{1}'$ and $w_{2}=\alpha_{2}w_{2}'$. Let $B=w_{1}(A)=w_{2}(A)$, a connected compact subset of $A$. Since $\alpha_{1}$ and $\alpha_{2}$ are injective, $B$ has non-zero diameter and thus $H^{1}(B)>0$. But $H^{1}(B)=H^{1}(\alpha_{1}w_{1}'(A)\cap\alpha_{2}w_{2}'(A))=0$, which is absurd. We conclude that $\langle\alpha_{1},\alpha_{2}\rangle$ is free of rank $2$.
\end{proof}

\begin{remark}
    The bound $c_{1}+c_{2}\leq1$ in Proposition \ref{pingpong} is sharp. Indeed, for $n\geq2$, let $c_{n}$ be the solution in $]\frac{1}{2},1[$ to the equation $x+x^{2}+\cdots+x^{n}=1$. Then $\alpha_{1}:\R\rightarrow\R,\;x\mapsto c_{n}x$ and $\alpha_{2}:\R\rightarrow\R,\;x\mapsto c_{n}x+1$ are injective contractions (similitudes, even) with common ratio $c_{n}$, but $\alpha_{1}\alpha_{2}^{n}=\alpha_{2}\alpha_{1}^{n}$. By increasing $n$, we get values of $c_{n}$ that are arbitrarily close to $\frac{1}{2}$.
\end{remark}

\section{Proof of the Theorems}
\label{sec2}

\paragraph{Heights.} Let $K$ be a finitely generated field, $V$ a projective variety over $K$ and $\mathcal{L}$ an ample line bundle on $V$. We also fix an algebraic closure $\bar{K}$ of $K$. If $K$ has positive characteristic, we may assume that it is infinite (otherwise every endomorphism over $K$ has the same set of preperiodic points and our theorems are vacuous). We then fix a \textit{Weil height function} $h_{\mathcal{L}}:V(\bar{K})\rightarrow\R$, whose construction is detailed in \cite[§2.4]{BG06}. If $K$ is of characteristic zero, we fix a \textit{Moriwaki height function} $h_{\mathcal{L}}:V(\bar{K})\rightarrow\R$, first defined in \cite{Mor00}. Note that if $K$ is a number field, the Moriwaki height coincides with the classical Weil height. In both cases, the function $h_{\mathcal{L}}$ satisfies two properties:
\begin{itemize}
    \item[]\textbf{Northcott property.} The set $\{x\in V(\bar{K})\;|\;h_{\mathcal{L}}(x)\leq a,\;[K(x):K]\leq b\}$ is finite for any $a,b>0$.
    \item[]\textbf{Functoriality.} If $f\in\End(V)$ is defined over $K$ and is polarized by $\mathcal{L}$ with algebraic degree $d$, then $d\cdot h_{\mathcal{L}}$ and $f^{*}h_{\mathcal{L}}$ differ by a bounded function.
\end{itemize}
These constructions and their properties are also summarized in \cite[§3.2]{BHPT24}. The discussion above motivates the introduction of
\[\mathcal{H}_{\mathcal{L}}:=\{h:V(\bar{K})\rightarrow\R\;|\;\|h-h_{\mathcal{L}}\|_{\infty}<+\infty\},\]
a complete metric space when endowed with the metric $d(h_{1},h_{2}):=\|h_{1}-h_{2}\|_{\infty}$. We also define, for any endomorphism $f$ polarized by $\mathcal{L}$ of algebraic degree $d$,
\[\alpha_{f}:\mathcal{H}_{\mathcal{L}}\rightarrow\mathcal{H}_{\mathcal{L}},\;h\mapsto\frac{1}{d}f^{*}h.\]
By the surjectivity of $f:V(\bar{K})\rightarrow V(\bar{K})$, we have
\[\forall h_{1},h_{2}\in \mathcal{H}_{\mathcal{L}},\quad \|\alpha_{f}(h_{1})-\alpha_{f}(h_{2})\|_{\infty}=\frac{1}{d}\|h_{1}-h_{2}\|_{\infty}\] so $\alpha_{f}$ is an injective contraction with ratio $\frac{1}{d}$. By Banach's fixed point theorem, $\alpha_{f}$ has a unique fixed point in $\mathcal{H}_{\mathcal{L}}$ which is called the \textit{canonical height} $h_{f}$ associated with $f$ and $\mathcal{L}$. It therefore satisfies $f^{*}h_{f}=d\cdot h_{f}$. By the Northcott property, we have $\{x\in V(\bar{K})\;|\;h_{f}(x)=0\}=\mathrm{PrePer}(f)$.

\begin{proof}[Proof of Theorem \ref{uni1}]
    We are given two endomorphisms $f_{1},f_{2}$ of a projective variety $V$, polarized by the same line bundle $\mathcal{L}$, say $f_{i}^{*}\mathcal{L}\cong\mathcal{L}^{\otimes d_{i}}$ with $d_{i}\geq2$, and with distinct sets of preperiodic points. We may fix a finitely generated field over which $f_{1},f_{2},V$ and $\mathcal{L}$ are defined. This allows us to define the space $\mathcal{H}_{\mathcal{L}}$, as well as the maps $\alpha_{1},\alpha_{2}:\mathcal{H}_{\mathcal{L}}\rightarrow\mathcal{H}_{\mathcal{L}}$ associated with $f_{1}$ and $f_{2}$. These are injective contractions with ratios $\frac{1}{d_{1}},\frac{1}{d_{2}}\leq\frac{1}{2}$ and fixed points $h_{f_{1}},h_{f_{2}}$.
    
    Note that the preperiodic points of a polarized endomorphism are isolated, as a consequence of \cite[Corollary 2.2]{Fak03}. In particular, the points in $\mathrm{PrePer}(f_{i})$ defined over our original field $\K$ coincide with the points in $\mathrm{PrePer}(f_{i})$ over $\bar{K}$. Now, since we assume that $\mathrm{PrePer}(f_{1})\neq\mathrm{PrePer}(f_{2})$, we obtain $h_{f_{1}}\neq h_{f_{2}}$. We can therefore apply Proposition \ref{pingpong}: the semigroup generated by $\alpha_{1}$ and $\alpha_{2}$ is free of rank $2$. Finally, any relation of the form $f_{i_{1}}\cdots f_{i_{n}}=f_{j_{1}}\cdots f_{j_{m}}$ in $\End(V)$ would imply a relation $\alpha_{i_{n}}\cdots \alpha_{i_{1}}=\alpha_{j_{m}}\cdots \alpha_{j_{1}}$. From this we conclude that $f_{1}$ and $f_{2}$ generate a free semigroup of rank $2$.
\end{proof}

In order to prove Theorem \ref{uni2}, we need to adapt \cite[Lemma 4.5]{BHPT24}. For any subset $F$ of endomorphisms of $V$ semi-polarized by the same line bundle, and any $d\geq1$, we denote by $F_{d}$ the set of elements of $F$ of algebraic degree $d$, and by $F_{\geq d}$ the set of elements of $F$ of algebraic degree at least $d$. 

\begin{lemma}
    \label{gen}
    Let $S$ be a subsemigroup of $\End(V)$, all of whose elements are semi-polarized by the same line bundle. Suppose there is a generating set $F$ of $S$ such that for any $\sigma,\tau\in F_{1}\cup\{\mathrm{id}\}$ and any $f,g\in F_{\geq2}$, we have $\mathrm{PrePer}(\sigma f)=\mathrm{PrePer}(\tau g)$. Then $\mathrm{PrePer}(w_{1})=\mathrm{PrePer}(w_{2})$ for any $w_{1},w_{2}\in S_{\geq2}$.
\end{lemma}

\begin{proof}
    We may assume that $S_{\geq2}$ is non-empty. Therefore $F_{\geq2}$ is also non-empty: we fix $f\in F_{\geq2}$ and define $P=\mathrm{PrePer}(f)$. For any $g\in F_{\geq2}$, we may set $\sigma=\tau=\mathrm{id}$ to obtain $P=\mathrm{PrePer}(f)=\mathrm{PrePer}(g)$. Similarly, for any $\sigma=F_{1}$, we have $\mathrm{PrePer}(\sigma f)=\mathrm{PrePer}(f)=P$. In particular, $P$ is both $f$ and $\sigma f$ invariant, so $\sigma (P)=\sigma f(P)=P$. Let $w\in S_{\geq2}$. Then $w$ is a product of elements of $F$, and thus $w(P)=P$.
    
    Fix a finitely generated field $K$ over which $S$ is defined, and consider the canonical height $h_{f}$ of $f$. We have $P=\{x\in V(\bar{K})\;|\;h_{f}(x)=0\}$, so by the Northcott property, all $w$-orbits in $P$ are finite. In other words, $P\subset \mathrm{PrePer}(w)$. As a result of \cite[Theorem 5.1]{Fak03}, $P$ is Zariski-dense in $V$. Hence, we can apply  \cite[Theorem 1.3, $(3)\Rightarrow (1)$]{YZ21a}, or \cite[Theorem A, $(3)\Rightarrow (4)$]{Car20} for the case of positive characteristic, to conclude that  $P=\mathrm{PrePer}(w)$.
\end{proof}

\begin{proof}[Proof of Theorem \ref{uni2}]
    Let $S$ be a finitely generated subsemigroup of $\End(V)$, all of whose elements are semi-polarized by the same line bundle. Let $F$ be any finite generating set of $S$. If $\mathrm{PrePer}(w_{1})\neq\mathrm{PrePer}(w_{2})$ for some $w_{1},w_{2}\in S_{\geq2}$, then by Lemma \ref{gen}, $\mathrm{PrePer}(\sigma f)\neq\mathrm{PrePer}(\tau g)$ for some $\sigma,\tau\in F_{1}\cup\{\mathrm{id}\}$ and $f,g\in F_{\geq2}$. Applying Theorem \ref{uni1}, we deduce that $\sigma f$ and $\tau g$ are independent. Since $\sigma f, \tau g\in F\cup F^{2}$, we have $\Delta(F)\leq2$ and thus $\Delta(S)\leq 2$. In particular, using inequality (\ref{eq1}),
    \[\Sigma(S)\geq\log(2)/\Delta(S)\geq\log(2)/2>0,\]
    so $S$ is of uniform exponential growth.
\end{proof}

In order to treat the linear case, we shall need the following result from E. Breuillard and T. Gelander.

\begin{theorem}[{\cite[Theorem 2.3]{BG05}}]
    \label{linear}
    Let $K$ be a finitely generated field, and let $n\geq1$. There exists a constant $c(n,K)<+\infty$ such that for any subset $F\subset \mathrm{GL}_{n}(K)$ that generates a non virtually nilpotent group, $\Delta(F)\leq c(n,K)$.
\end{theorem}

\begin{proof}[Proof of Theorem \ref{uni3}]
    Let $S$ be a finitely generated subsemigroup of $\End(\Ps^{1})$ which is not of polynomial growth. Fix a finitely generated field $K$ over which $S$ is defined. Since constant maps in $\End(\Ps^{1})$ are left-absorbing, they cannot be part of an independent pair. In other words, $\Delta(S)=\Delta(S_{\geq1})$, so we may assume that $S$ contains no non-constant maps.
    
    If $S_{\geq2}$ is non-emtpy, then by \cite[Proposition 4.10]{BHPT24}, $\mathrm{PrePer}(f)\neq\mathrm{PrePer}(g)$ for some $f,g\in S_{\geq2}$, and so by Theorem \ref{uni2} we have $\Delta(S)\leq2<+\infty$. Suppose then that $S$ only contains elements of degree $1$. Let $F$ be any finite generating set of $S$. The group generated by $F$ in $\mathrm{PGL}_{2}(K)$ is not of polynomial growth, so is not virtually nilpotent (\cite{Wol68}). By Theorem \ref{linear}, there exists a constant $c<+\infty$ that only depends on $K$ such that $\Delta(F)\leq c$. We conclude that $\Delta(S)\leq c<+\infty$.
    
    In particular, if $S$ is of exponential growth, then $\Sigma(S)\geq\log(2)/\Delta(S)>0$ so $S$ is of uniform exponential growth.
\end{proof}

\begin{remark}
    Note that in the case $S_{\geq2}\neq\varnothing$, the upper bound on the diameter of independence is absolute, whereas in the linear case, it depends on the choice of semigroup. Indeed, there is no uniform upper bound on $\Delta(S)$ for all semigroups $S\subset \mathrm{PGL_{2}}$ of exponential growth (see \cite[Theorem 1.7]{Bre06}; this phenomenon only occurs when $S$ generates a virtually solvable, non virtually nilpotent group). On the other hand, the existence of a uniform lower bound on $\Sigma(S)$ for all semigroups $S\subset \mathrm{PGL_{2}}$ of exponential growth is still unknown; in fact it would imply Lehmer's conjecture (see \cite[Question 1.8]{Bre06}).
\end{remark}

\bibliographystyle{amsalpha}
\bibliography{bib}

\providecommand{\bysame}{\leavevmode\hbox to3em{\hrulefill}\thinspace}
\providecommand{\MR}{\relax\ifhmode\unskip\space\fi MR }
\providecommand{\MRhref}[2]{%
  \href{http://www.ams.org/mathscinet-getitem?mr=#1}{#2}
}
\providecommand{\href}[2]{#2}
\begin{thebibliography}{BHPT24}

\bibitem[BG05]{BG05}
E.~Breuillard and T.~Gelander, \emph{Cheeger constant and algebraic entropy of
  linear groups}, Int. Math. Res. Not. \textbf{2005} (2005), no.~56,
  3511--3523.

\bibitem[BG06]{BG06}
E.~Bombieri and W.~Gubler, \emph{Heights in diophantine geometry}, Cambridge
  University Press, Cambridge, 2006.

\bibitem[BH85]{BH85}
M.~F. Barnsley and A.~N. Harrington, \emph{Mandelbrot set for pairs of linear
  maps}, Physica D: Nonlinear Phenomena \textbf{15} (1985), no.~3, 421--432.

\bibitem[BHPT24]{BHPT24}
J.~P. Bell, K.~Huang, W.~Peng, and T.~J. Tucker, \emph{A tits alternative for
  endormophisms of the projective line}, J. Eur. Math. Soc. \textbf{26} (2024),
  no.~12, 4903--4922.

\bibitem[Bre06]{Bre06}
E.~Breuillard, \emph{On uniform exponential growth for solvable groups}, 2006,
  arXiv preprint math/0602076.

\bibitem[Car20]{Car20}
A.~Carney, \emph{Heights and arithmetic dynamics over finitely generated
  fields}, 2020, arXiv preprint arXiv:2010.07200.

\bibitem[Fak03]{Fak03}
N.~Fakhruddin, \emph{Questions on self maps of algebraic varieties}, J.
  Ramanujan Math. Soc. \textbf{18} (2003), no.~2, 109--122.

\bibitem[Fal85]{Fal85}
K.~Falconer, \emph{The geometry of fractal sets}, Cambridge university press,
  Cambridge, 1985.

\bibitem[Mor00]{Mor00}
A.~Moriwaki, \emph{Arithmetic height functions over finitely generated fields},
  Invent. Math. \textbf{140} (2000), no.~1, 101--141.

\bibitem[Okn98]{Okn98}
J.~Okniński, \emph{Semigroups of matrices}, World Scientific, 1998.

\bibitem[Ser60]{Ser60}
J.-P. Serre, \emph{Analogues kählériens de certaines conjectures de weil},
  Ann. Math. \textbf{71} (1960), no.~2, 392--394.

\bibitem[Wol68]{Wol68}
J.~A. Wolf, \emph{Growth of finitely generated solvable groups and curvature of
  riemannian manifolds}, J. Differ. Geom. \textbf{2} (1968), 421--446.

\bibitem[YZ21]{YZ21a}
X.~Yuan and S.~Zhang, \emph{The arithmetic hodge index theorem for adelic line
  bundles ii: finitely generated fields}, 2021, arXiv preprint
  arXiv:1304.3539v2.

\bibitem[Zha95]{Zha95}
S.~Zhang, \emph{Small points and adelic metrics}, J. Algebr. Geom. \textbf{4}
  (1995), no.~2, 281--300.

\end{thebibliography}

\end{document}